\documentclass[final,12pt, 3p,times]{elsarticle}

\usepackage{amsmath,amssymb}
\usepackage{graphicx}
\usepackage{amsthm}
\usepackage{color}
\usepackage{slashbox}
\usepackage{siunitx}
\usepackage{here}
\usepackage{mathtools}

\newtheorem{theo}{Theorem}[section]

\newtheorem{cor}[theo]{Corollary}
\newtheorem{rem}[theo]{Remark}

\newcommand{\mc}{\mathcal}
\newcommand{\mb}{\mathbb}

\newcommand{\Honezero}{H^1_0(\Omega)}

\newcommand{\Ltwo}{L^{2}(\Omega)}
\newcommand{\red}[1]{\textcolor{red}{#1}}

\newcommand{\ra}{\rightarrow}
\newcommand{\f}{\frac}

\newcommand{\ball}{\overline{B}(\hat{u},\rho)}
\newcommand{\oballrho}{B(\hat{u},\rho)}

\newcommand{\ten}[1]{\times 10^{#1}}

\journal{Elsevier}

\begin{document}
	
	\begin{frontmatter}
		
		\title{Numerical verification method for positive solutions of elliptic problems}

		\author[Waseda-ims]{Kazuaki Tanaka}
		\cortext[mycorrespondingauthor]{Corresponding author}
		\ead{tanaka@ims.sci.waseda.ac.jp}
	
		\address[Waseda-ims]{Institute for Mathematical Science, Waseda University, 3-4-1, Okubo Shinjyuku-ku, Tokyo 169-8555, Japan}

		\begin{abstract}
			The purpose of this paper is to propose methods for verifying the positivity of a weak solution $ u $ of an elliptic problem assuming $ H^1_0 $-error estimation $ \left\|u-\hat{u}\right\|_{H_{0}^{1}} \leq \rho $ given some numerical approximation $ \hat{u} $ and an explicit error bound $ \rho $.
			We provide a sufficient condition for the solution to be positive and analyze the range of application of our method for elliptic problems with polynomial nonlinearities.
			We present numerical examples where our method is applied to some important problems.
		\end{abstract}
		
		
	\end{frontmatter}
	

	\section{Introduction}
	\label{sect:1}
	Over the last several decades, numerous studies have been conducted on the semilinear elliptic equations
	\begin{align}
	\label{mainpro}
	-\Delta u(x)=f(u(x)), ~~x \in \Omega,
	\end{align}
	with appropriate boundary conditions. One such example is the Dirichlet problem
	\begin{align}
	\label{mainpro:d}
	\left\{\begin{array}{l l}
	-\Delta u(x)=f(u(x)), &x \in \Omega,\\
	u(x)=0,  &x \in \partial\Omega,\\
	\end{array}\right.
	\end{align}
	where $\Omega \subset \mathbb{R}^{N}$~$(N=1,2,3,\cdots)$ is a bounded domain,
	$\Delta $ is the Laplacian,
	and $ f: \mb{R} \ra \mb{R}$ is a given nonlinear map.
	In particular, the investigation of positive solutions of \eqref{mainpro} has attracted significant attention \cite{lions1982existence,gidas1979symmetry,lin1994uniqueness,damascelli1999qualitative,gladiali2011bifurcation,de2019morse}.
	Here are some important examples derived from model problems for many applications in which we are interested.
	Positive solutions of problem \eqref{mainpro:d} with $ f(t) = \lambda t + t|t|^{p-1} $, $ \lambda \in [0,\lambda_1(\Omega)) $, $p\in (1,p^*)$ have been investigated from various points of view --- uniqueness, multiplicity, nondegeneracy, symmetricity, and so on \cite{gidas1979symmetry,lin1994uniqueness,damascelli1999qualitative,gladiali2011bifurcation,de2019morse},
	where
	$p^*=\infty$ when $N=1,2$ and $p^*=(N+2)/(N-2)$ when $ N\geq3 $, and
	$ \lambda_1(\Omega) $ is the first eigenvalue of $ -\Delta $ imposed on the homogeneous Dirichlet boundary value condition;
	the eigenvalue problem is understood in the weak sense at least when $ \Omega $ is not regular.
	Another important nonlinearity is $ f(t)=\varepsilon^{-2}(t-t^3) $, which corresponds to the stationary problem of the Allen-Cahn equation motivated by \cite{allen1979microscopic} and has been investigated by many researchers. The value $\varepsilon>0$ is a small parameter related to the so-called singular perturbation.
	A variational method ensures that problem \eqref{mainpro:d} with this nonlinearity has a positive solution when $ \varepsilon^{-2} \geq \lambda_1(\Omega)$.
	On the other hand, when $ \varepsilon^{-2} < \lambda_1(\Omega)$, no positive solution is admitted as we prove later.
	Despite these results, quantitative information about the positive solutions, such as their shape, has not been clarified analytically.
	Throughout this paper, $ H^k(\Omega) $ denotes the $k$-th order $ L^2 $ Sobolev space.
	We define $ \Honezero := \{ u \in H^1(\Omega) : u = 0~\mbox{on} ~ \partial \Omega\} $, with the inner product $(u, v)_{H^1_0}:=(\nabla u, \nabla v)_{L^2}$ and the norm $\| u \|_{H^1_0} := \sqrt{(u, u)_{\smash{H^1_0}}}$.
	We say that the solution $ u $ is ``positive'' if $ u>0 $ in $ \Omega $, and ``nonnegative'' if $ u\geq 0 $ in $ \Omega $.
	
	This paper is concerned with numerical verification (also known as verified numerical computation or computer-assisted proof) for positive weak solutions of problem \eqref{mainpro:d}.
 	The target weak form of \eqref{mainpro:d} will be shown explicitly at the beginning of the next section together with further regularity assumptions for nonlinearity $ f $.
	The pioneering research on numerical verification methods for partial differential equations began with \cite{nakao1988numerical,plum1991computer} and has been further developed by many researchers (see, for example, \cite{plum2008,nakao2011numerical,tanaka2017sharp} and the references therein).
	In particular, studies have revealed that methods based on several fixed point theorems for Newton operators (including their simplified versions) are greatly effective.
	This approach is closely related to our method (see Section \ref{sec:nk}).
	These methods enable us to obtain an explicit ball containing exact solutions of \eqref{mainpro:d}.
	For weak solutions, this is typically done in the sense of the norm $\left\|\cdot\right\|_{H_{0}^{1}}$, because $ \Honezero$ is a natural solution-space for \eqref{mainpro:d} in the distributional sense.
	In other words, they allow us to prove, for a numerical approximation $ \hat{u} \in \Honezero $, the existence of an exact weak solution $ u \in H^1_0(\Omega) $ satisfying 
	\begin{align}
	\label{h10error}
	\left\|u-\hat{u}\right\|_{H_{0}^{1}} \leq \rho
	\end{align}
	for an explicit error bound $ \rho $.
	Therefore, these methods have the advantage that quantitative information about solutions to a target equation is provided accurately in a strict mathematical sense.
	However, irrespective of how small the error bound is, the positivity of some solutions is not guaranteed without additional considerations.
	In particular, in the homogeneous Dirichlet case \eqref{mainpro:d}, it is possible for a solution that is verified by such methods to be negative near the boundary $\partial\Omega$.
	When $ N \leq 3 $, positivity can be verified if we have an $ H^3 $-norm evaluation $ \|u-\hat{u}\|_{H^3} $ with $\hat{u} \in H^3(\Omega)$ and an explicit bound for the embedding $ H^3(\Omega) \cap H^1_0(\Omega) \hookrightarrow C^1(\overline{\Omega})$.
	The bound can be evaluated by applying the bound for the embedding $ H^2(\Omega) \hookrightarrow C(\overline{\Omega})$ to $ u $ and its derivatives $\partial u/ \partial x_i$ ($i=1,2,\cdots,N$).	
	However, for some shapes of $ \Omega $, for example, a nonconvex polygonal domain, the regularity of the solution $ u $ is in general outside $ H^3(\Omega) $.
	Otherwise, even if $ u, \hat{u} \in  H^3(\Omega) $, evaluating $ \|u-\hat{u}\|_{H^3} $ itself is not easy, and troublesome numerical techniques for estimating the slope of $ \hat{u} $ are required to complete the proof of positivity.

	In previous studies, we developed the method of verifying the positivity of solutions of \eqref{mainpro:d} \cite{tanaka2015numerical,tanaka2017numerical,tanaka2017sharp}.
	These methods succeeded in verifying the existence of positive solutions by checking simple conditions, but they required $ L^{\infty} $-error estimation obtained by considering the embedding from a solution set with $ H^2 $-regularity.
	This is primarily for these methods because they need to find a subdomain where $ u $ may be negative and to evaluate the first eigenvalue of $ -\Delta $ on this subdomain.
	For the same reasons as mentioned above, this requirement narrows the range of applications of these methods.
	
	The main contribution of this paper is proposing methods for verifying the positivity of solutions $ u $ of \eqref{mainpro:d} while assuming only $ H^1_0 $-error estimation \eqref{h10error}, which is more generally applicable (wider class of domains, solutions that lack $ H^2 $-regularity, and so on) than previous methods \cite{tanaka2015numerical,tanaka2017numerical,tanaka2017sharp}.
	Theorems \ref{theo1} and \ref{theo:keeppositive}, and Corollary \ref{theo:linf} enable us to verify the nonnegativity of $ u $ under suitable conditions.
	The positivity of $ u $ follows from its nonnegativity using a maximum principle under appropriate conditions, such as when $ f $ is a subcritical polynomial (see, for example, \cite{drabek2009maximum}).
	Table \ref{table:applicable} summarizes the scope of application of our theorems to the case in which $ f $ is a subcritical polynomial $ f(t) = \lambda t + \sum_{i=2}^{n(<p^*)}a_{i}t|t|^{i-1}$, $ \lambda, a_i \in \mb{R}$, $ a_i \neq 0 $ for some $ i $.
	They are also applicable for more general nonlinearities other than polynomials (see Theorems \ref{theo1} and \ref{theo:keeppositive}, and Corollary \ref{theo:linf}).
	The table shows that the coefficient $ \lambda $ of the linear term essentially affects the existence of a positive solution of \eqref{mainpro:d} and the applicability of our theorems.
	In particular, there exists no positive solution in two specific cases. 
	The first case is when $ \lambda \geq \lambda_1(\Omega)$ and $a_i \geq 0$ for all $ i $.
	This can be checked by multiplying \eqref{mainpro} with the first eigenfunction of $ -\Delta $ and integrating both sides.
	The second case is when $ \lambda < \lambda_1(\Omega)$ and $a_i \leq 0$ for all $ i $.
	This can be checked in the same way.
	Our theorems are applicable to the other cases where the problem \eqref{mainpro:d} may admit positive solutions (see again Table \ref{table:applicable}).

	\renewcommand{\arraystretch}{1.3}
	\begin{table}[h]
		\label{table:applicable}
		\caption{The applicability of our theorems to a subcritical polynomial $ f(t) = \lambda t + \sum_{i=2}^{n(<p^*)}a_{i}t|t|^{i-1}$, $ \lambda, a_i \in \mb{R} $, $ a_i \neq 0 $ for some $ i $. Here $ \lambda_1(\Omega) $ is the first eigenvalue of $ -\Delta $ imposed on the homogeneous Dirichlet boundary value condition.}
		\begin{center}
			\begin{tabular}{l|p{30mm}|p{30mm}}
				\hline
				\backslashbox[30mm]{$a_i$}{$\lambda$}&$\geq \lambda_1(\Omega)$&$<\lambda_1(\Omega)$\\
				\hline
				\hline
				$a_i \geq 0$ for all $ i $ &No positive solution& Theorem \ref{theo1}\\
				\hline
				$a_i \leq 0$ for all $ i $ &Theorem \ref{theo:keeppositive}& No positive solution\\
				\hline
				$a_i a_j < 0$ for some $ i,j $ &Corollary \ref{theo:linf}& Theorem \ref{theo1}\\
				\hline
			\end{tabular}
		\end{center}
	\end{table}
	\renewcommand{\arraystretch}{1}
	
	Assuming a certain growth condition for $ f $, Theorem \ref{theo1} provides a sufficient condition on the nonnegativity of $ u $ that can be checked only from $ H^1_0 $-evaluation as in \eqref{h10error}.
	Theorem \ref{theo1} is proved by a constructive norm estimation for the negative part $ u_{-}:= \max\left\{-u,0\right\} $ of $ u $ by considering the embedding corresponding to each exponent.
	Theorem \ref{theo:keeppositive} verifies the nonnegativity of $ u $ by imposing another inequality condition on $ f $,
	which is proved using a fundamentally different approach based on the Newton iteration.
	It ensures, under an appropriate condition, that a Newton sequence staring from a nonnegative function $ \omega \not\equiv
	0 $ remains nonnegative at every step, and as a result confirms the nonnegativity of the solution $ u $ to which it converges.
	Known as Newton-Kantorovich's theorem \cite{deuflhard1979affine, kantorovich1982functional},
	the convergence property of Newton's method on function spaces has been investigated in several studies (see, for example, \cite{argyros2008convergence, proinov2010new, argyros2012weaker} for recent results).
	However, little is known about the influence of Newton operators on the sign of functions.
	In this sense, Theorem \ref{theo:keeppositive} increases our understanding of Newton's method for elliptic equations.
	Among the cases listed in Table \ref{table:applicable}, when $ \lambda \geq \lambda_1(\Omega)$ and $a_i a_j < 0$ for some pair $ \{i,j\} $, only Corollary \ref{theo:linf} is applicable. 
	In fact, this can be applied to all the cases listed in Table \ref{table:applicable}.
	Although Corollary \ref{theo:linf} is a generalized version of a previous theorem \cite[Theorem 2.2]{tanaka2017numerical}, this still requires $ L^{\infty} $-norm estimation; in this sense, a problem still remains.
	However, Theorems \ref{theo1} and \ref{theo:keeppositive} have wide application including several important problems such as those introduced at the beginning of this section.

	The remainder of this paper is organized as follows.
	In Section \ref{sec:1}, we provide a method for proving the positivity of a solution $ u $ of \eqref{mainpro:d} whose existence is confirmed as in inequality \eqref{h10error}.
	This method does not restrict verification methods, but admits any methods that can prove the existence of a solution with $ H^1_0 $-error estimation.
	Section \ref{sec:nk} provides another positivity-validation method based on the Newton iteration that retains nonnegativity.
	Finally, in Section \ref{sec:ex}, we present numerical examples where our method is applied to elliptic problems with the nonlinearities introduced at the beginning of this section.
	
	\section{Verification of positivity --- Constructive norm estimation for $ u_{-} $ when $\lambda < \lambda_1(\Omega)$}\label{sec:1}
	
	We begin by introducing some required notation.
	We denote $ V=\Honezero $ and $ V^*=$ (the topological dual of $ V $).
	For two Banach spaces $X$ and $Y$, the set of bounded linear operators from $X$ to $Y$ is denoted by ${\mc L}(X, Y)$ with the usual supremum norm $\| T \|_{{\mc L}(X, Y)} := \sup\{ \| T u \|_{Y} / \| u \|_{X} : {u \in X \setminus \{0\}}\}  $ for $T \in {\mc L}(X, Y)$.
	The norm bound for the embedding $V\hookrightarrow L^{p+1}\left(\Omega\right)$ is denoted by $C_{p+1}$, that is, $C_{p+1}$ is a positive number that satisfies
	\begin{align}
	\label{embedding}
	\left\|u\right\|_{L^{p+1}(\Omega)}\leq C_{p+1}\left\|u\right\|_{V}~~~{\rm for~all}~u\in V,
	\end{align}
	where $p\in [1,\infty)$ when $N=1,2$ and $p\in [1,p^*]$ when $ N\geq3 $.
	Assuming that $ f $ is a $ C^1 $ function satisfying
	\begin{align*}
		&|f(t)| \leq a_0 |t|^p + b_0 \text{~~~for~all~~} t \in \mb{R}\\
		&|f'(t)| \leq a_1 |t|^{p-1} + b_1 \text{~~~for~all~~} t \in \mb{R}
	\end{align*}
	for some $ a_0,a_1,b_0,b_1\geq 0 $ and $ p<p^* $,
	we define the operator $ F $ by
	\begin{align*}
		F : \left\{\begin{array}{ccc}{u(\cdot)} & {\mapsto} & {f(u(\cdot))}, \\
		{V} & {\rightarrow} & {V^{*}}.\end{array}\right.
	\end{align*}
	We then define another operator $ \mathcal{F} : V \rightarrow V^{*}$ by $\mathcal{F}(u) :=-\Delta u-F(u) $.
	More precisely, $ \mathcal{F} $ is characterized by
	\begin{align}
		\left<\mathcal F(u),v\right> = \left(\nabla u,\nabla v\right)_{L^2} - \left<F(u),v\right>  \text{~~for~all~~} u,v \in V,
	\end{align}
	where $\left<F(u),v\right> = \int_{\Omega} f(u(x)) v(x) dx$.
	The Fr\'echet derivatives of $ F $ and $ \mc{F} $ at $ \varphi \in V $ (respectively denoted by $ {F'_{\varphi}} $ and $ {\mc F'_{\varphi}} $) are given by
	\begin{align}
		&\langle F'_{\varphi}u,v\rangle = \int_{\Omega} f'(\varphi(x))u(x) v(x) dx \text{~~for~all~~} u,v \in V
		\shortintertext{and}
		&\langle \mathcal F'_{\varphi}u,v \rangle = \left(\nabla u,\nabla v\right)_{L^2} - \langle F'_{\varphi}u,v \rangle  \text{~~for~all~~} u,v \in V. \label{def:derivativecalf}
	\end{align}
	Under this notation and assumption, we look for positive solutions $ u \in V $ of 
	\begin{align}
		\label{main:fpro}
		\mathcal{F}(u)=0,
	\end{align}
	which corresponds to the weak form of \eqref{mainpro:d}.
	We assume that some verification method succeeds in proving the existence of a solution $u \in V$ of \eqref{main:fpro} in $\ball := \left\{v\in V : \left\|v-\hat{u}\right\|_{V}\leq\rho\right\}$ for some $ \hat{u} \in V $ and $ \rho >0 $.
	
	When $\lambda < \lambda_1(\Omega)$, the following theorem is helpful for verifying the nonnegativity of $ u $.

	\begin{theo}\label{theo1}
		Let $ f $ satisfy
		\begin{align}
		\label{f:theo1}
		-f(-t)\leq \lambda t + \displaystyle \sum_{i=1}^{n}a_{i}t^{p_{i}} \text{~~for~all~~}t\geq 0
		\end{align}
		for some $ \lambda < \lambda_1(\Omega) $, nonnegative coefficients $ a_1,a_2,\cdots,a_n $, and subcritical exponents $ p_1,p_2,\cdots,p_n \in (1,p^*)$.
		If	
		\begin{align}
		\label{cond:theo1}
		\displaystyle \sum_{i=1}^{n}a_i C_{p_{i}+1}^{2}\left( \left\|\hat{u}_{-}\right\|_{L^{p_{i}+1}}+C_{p_{i}+1}\rho\right)^{p_{i}-1}<1 - \f{\lambda}{\lambda_1(\Omega)},
		\end{align}
		then the verified solution $u \in V$ of \eqref{main:fpro} in $\ball$ is nonnegative.
	\end{theo}
	
	\begin{rem}
		\label{rem1:theo1}
		The polynomial
		$ f(t) = \lambda t + \sum_{i=2}^{n(<p^*)}a_{i}t|t|^{i-1} $ 
		with $ \lambda < \lambda_1(\Omega) $ and $a_i \in \mb{R} $, which were discussed in the previous section, obviously satisfies the required inequality \eqref{f:theo1}.
		Indeed, for the set of subscripts $ \Lambda_{+} $ for which $ a_i \geq 0~(i \in \Lambda_{+})$ and $ a_i < 0$ $($otherwise$)$,
		we have $-f(-t)\leq \lambda t + \sum_{i \in \Lambda_{+}}a_{i}t^{i} \text{~~for~all~~}t\geq 0 $.
		Moreover, for this polynomial, the positivity of $ u $ follows from the nonnegativity via the maximum principle $ ( $see, for example, {\rm\cite{drabek2009maximum}} for a generalized maximum principle applicable for weak solutions$ ) $.
	\end{rem}
	
	\begin{rem}
		The formula in parentheses in \eqref{cond:theo1} goes to $0$ as $\hat{u}$ approaches a nonnegative function and $ \rho \downarrow 0 $.
		Therefore, as long as verification succeeds for a nonnegative approximation $\hat{u}$ with sufficient accuracy,
		the nonnegativity of $ u $ can be confirmed using Theorem $\ref{theo1}$.
	\end{rem}

	\begin{rem}
		Even if the approximation $ \hat{u} $ is negative in some parts of $ \Omega $, if it is close enough to a nonnegative function in the sense that $\ball$ contains at least one nonnegative function,
		this theorem may work for verifying nonnegativity because it only requires the bound for $ \left\|\hat{u}_{-}\right\|_{L^{p_{i}+1}} $ small enough to satisfy \eqref{cond:theo1}.
		For the same reason, this theorem is applicable for $ \hat{u} $ whose nonnegativity is difficult to prove.
		For example, this theorem is reasonable even when long computation time is required for proving the nonnegativity of $ \hat{u} $ due to its high regularity (see Section {\rm\ref{sec:ex}}).
	\end{rem}

	\subsection*{Proof of Theorem \ref{theo1}}
	First we prove that, for $p\in (1,p^*)$,
	\begin{align}
	\left\|u_{-}\right\|_{L^{p+1}}
	\leq
	\left\|\hat{u}_{-}\right\|_{L^{p+1}}+ C_{p+1}\rho.
	\label{uandu_}
	\end{align}
	We express $u\in V$ as $ u=\hat{u}+ \rho\omega$, where $\omega\in V$ satisfies $\left\|\omega\right\|_{V}\leq 1$.
	Because $b\displaystyle \geq(a-b)_{-}(:=\max\{-(a-b),0\})$ for nonnegative numbers $a,b\in \mathbb{R}$, we have
	\begin{align*}
	0 \leq u_{-}=&\left(\hat{u}+ \rho\omega\right)_{-}
	=\left(\hat{u}_{+}-\hat{u}_{-}+ \rho\omega_{+}- \rho\omega_{-}\right)_{-}
	=\left(\hat{u}_{+}+ \rho\omega_{+}-\left(\hat{u}_{-}+ \rho\omega_{-}\right)\right)_{-}
	\leq \hat{u}_{-}+ \rho\omega_{-},
	\end{align*}
	which implies \eqref{uandu_} because $\left\|\omega_{-}\right\|_{L^{p+1}}\leq C_{p+1}\left\|\omega_{-}\right\|_{V}\leq C_{p+1}$.
	
	We then prove that the norm of $ u_{-} $ vanishes.
	Because $u$ satisfies
	\begin{align*}
	\left(\nabla u,\nabla v\right)_{L^2}=\left<F(u),v\right>{\rm~~for~all~}v\in V,
	\end{align*}
	by fixing $v=u_{-}$, we have
	\begin{align}
	\left\|u_{-}\right\|_{V}^{2} 
	\leq &\displaystyle \int_{\Omega}\left\{\lambda\left(u_{-}(x)\right)^{2}+\sum_{i=1}^{n}a_{i}\left(u_{-}(x)\right)^{p_{i}+1}\right\}dx \nonumber \\
	= &\displaystyle \lambda\left\|u_{-}\right\|_{L^{2}}^{2}+\sum_{i=1}^{n}a_{i}\left\|u_{-}\right\|_{L^{p_{i}+1}}^{p_{i}+1}\nonumber \\
	\leq &\left\{ \displaystyle \frac{\lambda}{\lambda_{1}(\Omega)} +\sum_{i=1}^{n}a_{i}C_{p_{i}+1}^{2} \left\|u_{-}\right\|_{L^{p_{i}+1}}^{p_{i}-1} \right\} \left\|u_{-}\right\|_{V}^{2}.
	\end{align}
	Inequalities \eqref{cond:theo1} and \eqref{uandu_} lead to
	\begin{align}
	\displaystyle \frac{\lambda}{\lambda_{1}(\Omega)} +\sum_{i=1}^{n}a_{i}C_{p_{i}+1}^{2} \left\|u_{-}\right\|_{L^{p_{i}+1}}^{p_{i}-1} < 1,
	\end{align}
	which ensures $ \|u_{-}\|_V = 0 $.
	Therefore, the nonnegativity of $ u $ is proved.
	\qed
	
	\medskip
	This theorem can be applied to the nonlinearity discussed at the beginning of Section \ref{sect:1} (see again Remark \ref{rem1:theo1}).
	
	\begin{cor}
		\label{cor:emden}
		Let $ f(t)= \lambda t + t|t|^{p-1} $, with $ \lambda < \lambda_1(\Omega) $ and $ p \in (1,p^*) $.
		If 
		\begin{align}
		C_{p+1}^{2}\left(\left\|\hat{u}_{-}\right\|_{L^{p+1}}+C_{p+1}\rho\right)^{p-1}<1 - \f{\lambda}{\lambda_1(\Omega)},
		\end{align}
		then the verified solution $u \in V$ of \eqref{main:fpro} in $\ball$ is positive.
	\end{cor}
	
	\section{Verification of positivity --- Newton iteration retaining nonnegativity when $\lambda \geq \lambda_1(\Omega)$}
	\label{sec:nk}
	In this section, we discuss another approach to verifying the positivity of a solution $ u $ of \eqref{main:fpro} when $ \lambda \geq \lambda_1(\Omega) $ to which Theorem \ref{theo1} is not applicable.
	For this purpose, we impose another assumption on the nonlinearity $ f $:
	\begin{align}
	\label{asm:f2}
	f(t)\geq f'(t)t ~~~\mbox{for~all}~~~ t\geq 0.
	\end{align}
	The class of functions $ f $ satisfying \eqref{asm:f2} includes polynomials, which are ignored in Theorem \ref{theo1}, of the form
	\begin{align}
	\label{ft:allen}
	f(t)=\lambda t+\displaystyle \sum_{i=2}^{n(<p^*)}a_{i}t|t|^{i-1}~~\mbox{with}~~\lambda \geq \lambda_1(\Omega)~~\mbox{and}~~a_2,\,a_3,\cdots,\,a_n \le 0.
	\end{align}
	This admits some other important cases such as $ f(t) = \varepsilon^{-2}(t-t^3)$ with $ \varepsilon>0 $.
	Recall that there exists no positive solution when $ \lambda < \lambda_1(\Omega) $ and the coefficients $ a_i $ corresponding to super-linearity are nonpositive.
	
	The method proposed in this section is based on the Newton iteration retaining nonnegativity.
	Before presenting the main theorem of this section, we introduce the affine invariant Newton-Kantorovitch theorem, which ensures the convergence of Newton's method for a ``good'' starting point $ \hat{u} $.
	This theorem has wide applicability to verification methods for nonlinear equations, including differential equations.
	We discuss the positivity of a solution $ u $ whose existence is proved via this theorem.
	\begin{theo}[\cite{deuflhard1979affine}]
		\label{theo:nk}
		Let $ \hat{u} \in V$ be some approximation of a solution $ u $ of $\mc F(u)=0$.
		Suppose that there exists some $\alpha>0$ satisfying
		\begin{align}
		\label{alpha}
		||{\mc F'_{\hat{u}}}^{-1}\mc F(\hat{u})||_V\le\alpha.
		\end{align}
		Moreover, suppose that there exists some $\beta>0$ satisfying
		\begin{align}
		\label{beta}
		||{\mc  F'_{\hat{u}}}^{-1}(\mc  F'_v-\mc  F'_w)||_{{\cal L}(V,V)}\le\beta||v-w||_V~~\mathrm{for~all~} v,w \in D,
		\end{align}
		where $D=B(\hat{u},2\alpha+\delta)$ is an open ball depending on the above value $\alpha>0$ for small $\delta>0$.
		If
		\begin{align}
		\label{alphabeta}
		\alpha\beta \le \frac{1}{2},
		\end{align}
		then there exists a solution $u \in V$ of $\mc F(u)=0$ in $\ball$ with 
		\begin{align}
		\label{nkerror}
		\rho = \frac{1-\sqrt{1-2\alpha\beta}}{\beta}.
		\end{align}
		Furthermore, $ \mc  F'_{\varphi} $ is invertible for every $ \varphi \in \oballrho$, and the solution $u$ is unique in $\overline{B}(\hat{u},2\alpha)$.
	\end{theo}
	
	The following theorem verifies the nonnegativity of $ u $ the local existence of which is confirmed by Theorem \ref{theo:nk}.
	Here, $ g \in V^* $ is called nonnegative if and only if
	\begin{align}
	\label{def:positive}
	\langle g, \varphi\rangle \geq 0 ~~ \text { for all } \varphi \in V \text { with } \varphi \geq 0.
	\end{align}
	\begin{theo}\label{theo:keeppositive}
		Suppose the following:
		\begin{enumerate}
			\setlength{\parskip}{0cm}
			\setlength{\itemsep}{0cm}
			\item $ f $ satisfies \eqref{asm:f2};
			\item there exist $ \alpha $ and $ \beta $ satisfying \eqref{alpha}, \eqref{beta}, and \eqref{alphabeta};
			\item the minimal eigenvalue $ \mu_1(\hat{u}) $ of ${\cal F}_{\hat{u}}'$ is positive, where $ \mu_1(\hat{u}) $ is given by
			\begin{align}
				\label{def:eigenvalue}
				\mu_1(\hat{u}) = \inf_{v\in V\backslash{\{0\}}} \frac{\langle {\cal F}_{\hat{u}}'v,v \rangle}{\|v\|_{L^2}^2};
			\end{align}	
			\item there exists a nonnegative function $ \omega \not\equiv
			0 $  in $ \oballrho $.
		\end{enumerate}
		Then there exists a nonnegative solution $u \in V$ of $\mc F(u)=0$ in $\overline{B}(\hat{u},\rho)$.
	\end{theo}

	\begin{rem}
	Assumption $ 2 $ can be replaced with another condition which proves the convergence of Newton's method in a neighborhood around $ \hat{u} $ and the invertibility $ \mc  F' $ in the whole of it.
	For example, in {\rm \cite{nakao2011numerical}} and the references therein, numerical verification methods using several fixed-point theorems were developed.
	By applying an appropriate fixed point theorem such as Banach's fixed point theorem to the Newton operators starting from $ \hat{u} $ on the basis of these methods, another (but probably similar) condition can be used in place of Assumption $ 2 $.
	Note that the conditions proving the convergence of simplified Newton's method, such as in {\rm \cite{plum2008,nakao2011numerical}}, are not directly replaceable with Assumption $ 2 $, because Theorem $ \ref{theo:keeppositive} $ requires the convergence property of the original Newton's method.
	\end{rem}
	
	\begin{rem}
		Although it is unknown whether Assumption $ 3 $ may hold for all nonlinearities $ f $ satisfying \eqref{asm:f2} for some $ \hat{u} $ near to a positive function,
		Assumption $ 3 $ is expected to be satisfied at least for the nonlinearity \eqref{ft:allen} with $ \hat{u} $ approximating a positive solution of \eqref{main:fpro} with sufficient accuracy because a standard variational method ensures that the desired positive solutions are least-energy solutions.
	\end{rem}
	
	\begin{rem}
		Assumption $ 4 $ must hold for us to find a nonnegative or positive solution in $\ball$.
		In practice, it is useful for us to check $ \max\left\{\hat{u},0\right\} \in \oballrho$.
	\end{rem}

	\subsection*{Proof of Theorem \ref{theo:keeppositive}}
	Theorem \ref{theo:nk} and Assumption 2 guarantee the existence of a solution $ u $ in $ \overline{B}(\hat{u}, \rho) $.
	Therefore, it remains to prove the nonnegativity of $ u $.
	
	Assumption 3 ensures that the minimal eigenvalue $ \mu_1(\varphi) $ of $ \mc{F}'_{\varphi} $ is positive for all $ \varphi \in \oballrho $. 
	Indeed, if $ \mu_1(\varphi) $ is nonpositive for some $ \varphi \in \oballrho $, then it follows that there exists $ \varphi_0 \in \oballrho $ such that $ \mu_1(\varphi_0) = 0 $,
	because $ \mu_1(\varphi) $ is continuous with respect to $ \varphi $.
	This contradicts the result from Theorem \ref{theo:nk}.
	It can be proved from the following discussion that the operator $ \mc{F}_{\varphi}^{\prime-1} $ retains nonnegativity for all $ \varphi \in \oballrho $.
	Let $ w \in V $ satisfy $ \mc{F}'_{\varphi} w \geq 0$ in the sense of \eqref{def:positive}.
	By fixing $ v = w_{-}\,(= \max\{-w,0\}) $ in \eqref{def:derivativecalf}, we have
	\begin{align*}
		0 \leq \langle \mc{F}'_{\varphi} w, w_{-} \rangle
		= - \left\|w_{-}\right\|_{V}^{2} + \int_{\Omega} f'(\varphi(x))w^2_{-}(x) dx
		= - \langle \mc{F}'_{\varphi} w_{-}, w_{-} \rangle.
	\end{align*}
	Therefore, from the definition \eqref{def:eigenvalue} of $ \mu_1(\varphi) $, we have
	\begin{align*}
		\left\|w_{-}\right\|_{L^2}^{2} \mu_1(\varphi) \leq \langle \mc{F}'_{\varphi} w_{-}, w_{-} \rangle \leq 0.
	\end{align*}
	The positivity of $ \mu_1(\varphi) $ ensures that $\left\|w_{-}\right\|_{L^2} = 0$.
	Hence, $ w $ is nonnegative.

	We next prove that the Newton iteration in $ \oballrho $ retains nonnegativity, that is, for every nonnegative $ \varphi \in \oballrho $, the Newton operator
	\begin{align}
	T(\varphi) := \varphi - \mc{F}_{\varphi}^{\prime-1}\mc{F}(\varphi)	
	\end{align}
	on $ \oballrho $ maintains nonnegativity.
	For this operator, we have
	\begin{align*}
	T(\varphi)=\varphi-\mathcal{F}_{\varphi}^{\prime-1}\mathcal{F}(\varphi)
	=\varphi-\mathcal{F}_{\varphi}^{\prime-1}\left(\mathcal{F}_{\varphi}'\varphi+F_{\varphi}'\varphi-F(\varphi)\right)
	=\mathcal{F}_{\varphi}^{\prime-1}\left(F(\varphi)-F_{\varphi}'\varphi\right).
	\end{align*}
	Because Assumption 1 ensures that $ F(\varphi)-F_{\varphi}'\varphi \geq 0 $ for $ \varphi \geq 0 $ in the sense of \eqref{def:positive}, it follows that $ T(\varphi) \geq 0$ for nonnegative $ \varphi \in \oballrho $ (recall that $ \mc{F}_{\varphi}^{\prime-1} $ retains nonnegativity).
	Therefore, Assumption 4 ensures the existence of a Newton sequence starting from nonnegative $ \omega \not\equiv 0 $ that converges to a solution $ u \geq0 $.
	\qed
	
	\medskip
	This theorem can be applied to the special case discussed at the beginning of Section \ref{sect:1}.
	\begin{cor}
		\label{cor:allen}
		Let $f(t)=\varepsilon^{-2}(t-t^3)$, with $ \varepsilon^{-2} \geq \lambda_1(\Omega) $.
		If Assumptions $2,3,$ and $4$ in Theorem $\ref{theo:keeppositive}$ hold, then there exists a positive solution $u \in V$ of $\mc F(u)=0$ in $\overline{B}(\hat{u},\rho)$.
	\end{cor}
	\begin{proof}
		The polynomial $ f $ satisfies \eqref{asm:f2}.
		Therefore, Theorem \ref{theo:keeppositive} ensures the nonnegativity of $ u $.
		Its positivity follows from the maximum principle (see, for example, {\rm\cite{drabek2009maximum}}).
	\end{proof}
	
	\section{Numerical examples}\label{sec:ex}
	In this section, we present examples in which the positivity of solutions to \eqref{main:fpro} are verified via our method.
	All computations were implemented on a computer with 2.90 GHz Intel Core(TM) i9-7920X CPU, 128 GB RAM, and Ubuntu 18.04 using MATLAB 2018a with GCC version 6.3.0.
	All rounding errors were strictly estimated using the toolboxes INTLAB version 10.2 \cite{rump1999book} and kv library version 0.4.47 \cite{kashiwagikv}.
	We constructed approximate solutions of \eqref{main:fpro} for $ \Omega=(0,1)^2 $ from a Legendre polynomial basis.
	More concretely, we constructed $\hat{u}$ as
	\begin{align}
	\displaystyle \hat{u}(x,y)=\sum_{i=1}^{N}\sum_{j=1}^{N}u_{i,j}\phi_{i}(x)\phi_{j}(y),~~u_{i,j}\in \mathbb{R},
	\end{align}
	where each $\phi_{n}$ ($ n=1,2,3,\cdots $) is defined by
	\begin{align}
	\displaystyle \phi_{n}(x)=\frac{1}{n(n+1)}x(1-x)\frac{dQ_{n}}{dx}(x) \text{~~with~~}
	Q_{n}=\displaystyle \frac{(-1)^{n}}{n!}\left(\frac{d}{dx}\right)^{n}x^{n}(1-x)^{n},~~n=1,2,3,\cdots.
	\end{align}
	We define a finite dimensional subspace $ V_N~(\subset V) $ as the tensor product $V_N = \text{span\,}\{\phi_{1},\phi_{2},\cdots,\phi_{N}\} \otimes \text{span\,}\{\phi_{1},\phi_{2},\cdots,\phi_{N}\} $, then defining the orthogonal projection $ P_N $ from $ V $ to $ V_N $ by
	\begin{align*}
		(v- P_N v,v_N)_V = 0 \text{~~~for all~} v \in V \text{~and~} v_N \in V_N.
	\end{align*}
	We used \cite[Theorem 2.3]{kimura1999on} to obtain an explicit interpolation-error constant $ C_N $ satisfying
	\begin{align}
		\label{eq:interpolation}
		\left\|v-P_{N}v\right\|_{V}\leq C_{N}\left\|\Delta v\right\|_{L^{2}} \text{~~for~all~~} v \in V \cap H^2(\Omega).
	\end{align}
	Recall that our method does not limit the basis functions that constitute approximate solutions,
	being applicable to many kinds of bases other than the Legendre polynomial basis, such as the piecewise linear finite element basis or the Fourier basis, etc.
	
	We proved the existence of solutions $u$ of \eqref{main:fpro} in $ \ball $ using Theorem \ref{theo:nk}.
	The key constants $ \alpha $ and $ \beta $ were estimated by
	\begin{align*}
	\alpha \leq \|\mathcal{F}_{\hat{u}}^{\prime-1}\|_{\mathcal{L}\left(V^{*}, V\right)}\|\mathcal{F}(\hat{u})\|_{V^{*}} \text{~~~and~~~} 
	\beta \leq \|\mathcal{F}_{\hat{u}}^{\prime-1}\|_{\mathcal{L}\left(V^{*}, V\right)} L,
	\end{align*}
	where $ L $ is a positive number satisfying
	\begin{align*}
	\left\|F_{v}^{\prime}-F_{w}^{\prime}\right\|_{\mathcal{L}\left(V, V^{*}\right)} \leq L\|v-w\|_{V} \text {~~for all~~} v, w \in D.
	\end{align*}
	In the following examples,
	the inverse norm $ \|\mathcal{F}_{\hat{u}}^{\prime-1}\|_{\mathcal{L}\left(V^{*}, V\right)} $ was estimated using the method described in \cite{tanaka2014verified, liu2015framework} in the finite dimensional subspace $ V_N $.
	Moreover, we evaluated the upper bound on $ \|\mathcal{F}(\hat{u})\|_{V^{*}} $ by $C_2 \|\mathcal{F}(\hat{u})\|_{L^2} $,
	where $C_2$ is the constant of embedding $ \Ltwo \hookrightarrow V^* $ which in fact coincides with the constant of embedding $ V \hookrightarrow \Ltwo $ (see, for example, \cite{plum2008}).	
	This $ L^2$-norm was computed using a numerical integration method with strict estimation of rounding errors using \cite{kashiwagikv}.
	The embedding constant $ C_2 $ was calculated as $ C_2=(2 \pi^2)^{-\f{1}{2}} \approx 0.2251$ with strict estimation of rounding errors.
	Other embedding constants $ C_{p+1}~(p>1) $ were evaluated via the formula in \cite[Corollary A.2]{tanaka2017sharp}.
	
	For our first example, 
	we consider the problem of finding positive solutions to Emden's equation
	\begin{align}
	\left\{\begin{array}{l l}
	-\Delta u=u\left|u\right|^{p-1} &\mathrm{in~} \Omega,\\
	u=0 &\mathrm{on~} \partial\Omega
	\end{array}\right.\label{emden}
	\end{align}
	with $p=3, 5$,
	the approximate solutions of which are displayed in Fig.~\ref{fig1}.
	The Lipschitz constant $ L $ was estimated as
	\begin{align}
	\label{lip:emden}
	L \leq p(p-1) C_{p+1}^{3}\left(\|\hat{u}\|_{L^{p+1}}+C_{p+1} r\right)^{p-2},~~r=2\alpha+\delta~~{\rm for~small}~~\delta>0
	\end{align}
	via a simple calculation from the definition, where we set $ r $ to be the next floating-point number of $ 2\alpha $.
	We verified the positivity of the verified solutions $ u $ using Corollary \ref{cor:emden} (or Theorem \ref{theo1}).
	Table \ref{tab1} shows the verification result. 
	In all cases, the positivity of the verified solutions $ u $ were confirmed under the condition $ C_{p+1}^{2}\left(\left\|\hat{u}_{-}\right\|_{L^{p+1}}+C_{p+1}\rho\right)^{p-1} < 1 $.
	It should be noted that the positivity of the approximation $ \hat{u} $ was not proved but, in stead of it, upper bounds for $ \left\|\hat{u}_{-}\right\|_{L^{p+1}} $ were roughly estimated by dividing the domain $ \Omega $ into $ 2^{14} $ smaller congruent squares and implementing interval arithmetic on each of them.
	Whereas one can infer from the shapes of $ \hat{u} $ displayed in Fig.~\ref{fig1} that $ \hat{u} $ ($ p=3,5 $) are positive in $ \Omega $,
	we only used the rough estimation of the negative part to avoid proof of the positivity,
	because it requires much computational cost such as the rigorous calculation for the slope of $ \hat{u} $ near the boundary $ \partial \Omega $.
	
	\newcommand{\sizee}{0.49\hsize}
	\newcommand{\vminus}{\vspace{0mm}}
	\begin{figure}[H]
		\begin{minipage}{\sizee}
			\begin{center}
				\vminus\includegraphics[height=45 mm]{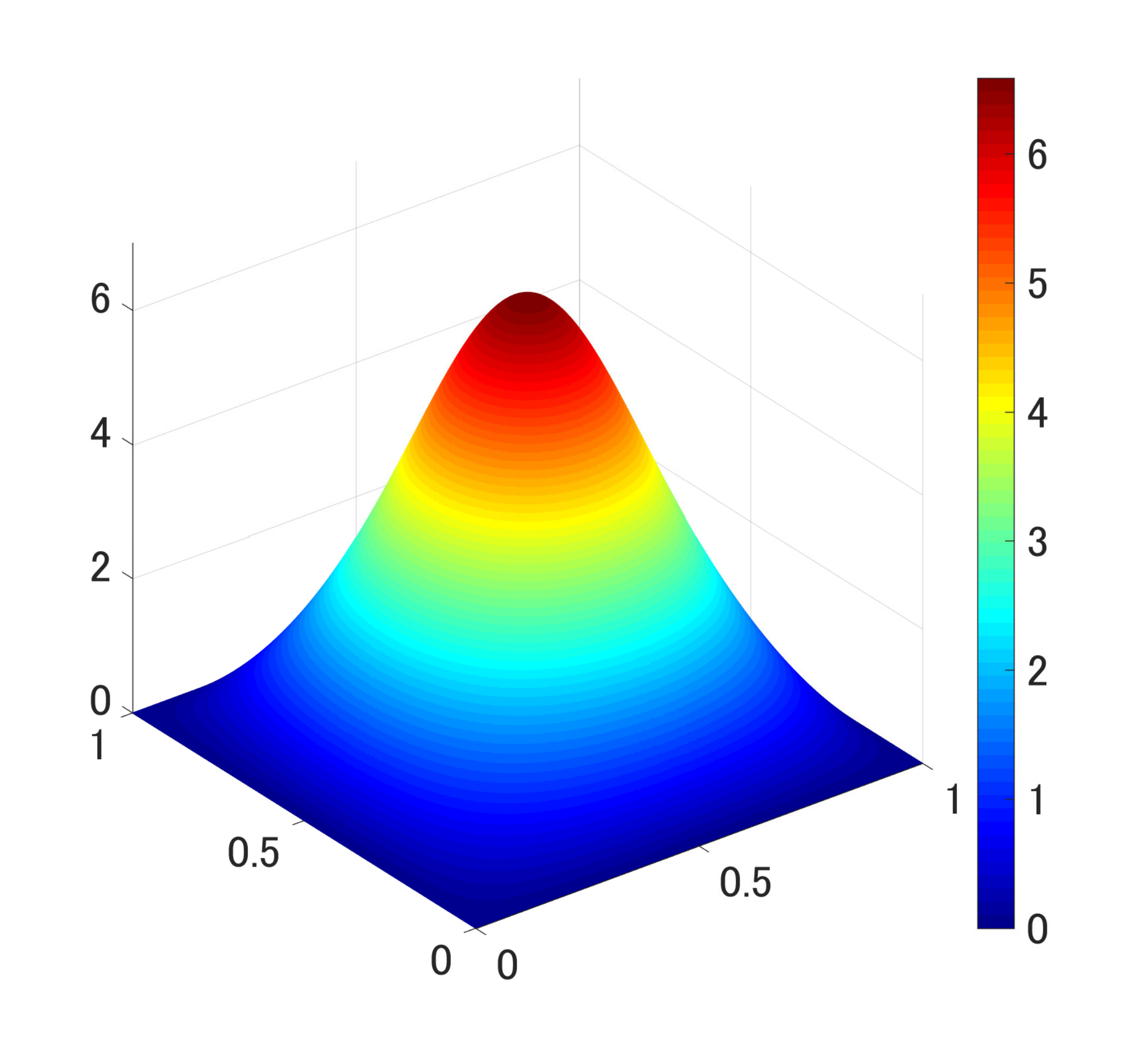}\\
				\footnotesize$p=3$,~~$\displaystyle \max_{x\in\Omega}\hat{u}(x)\approx 6.6232$
			\end{center}
			~
		\end{minipage}
		\begin{minipage}{\sizee}
			\begin{center}
				\vminus\includegraphics[height=45 mm]{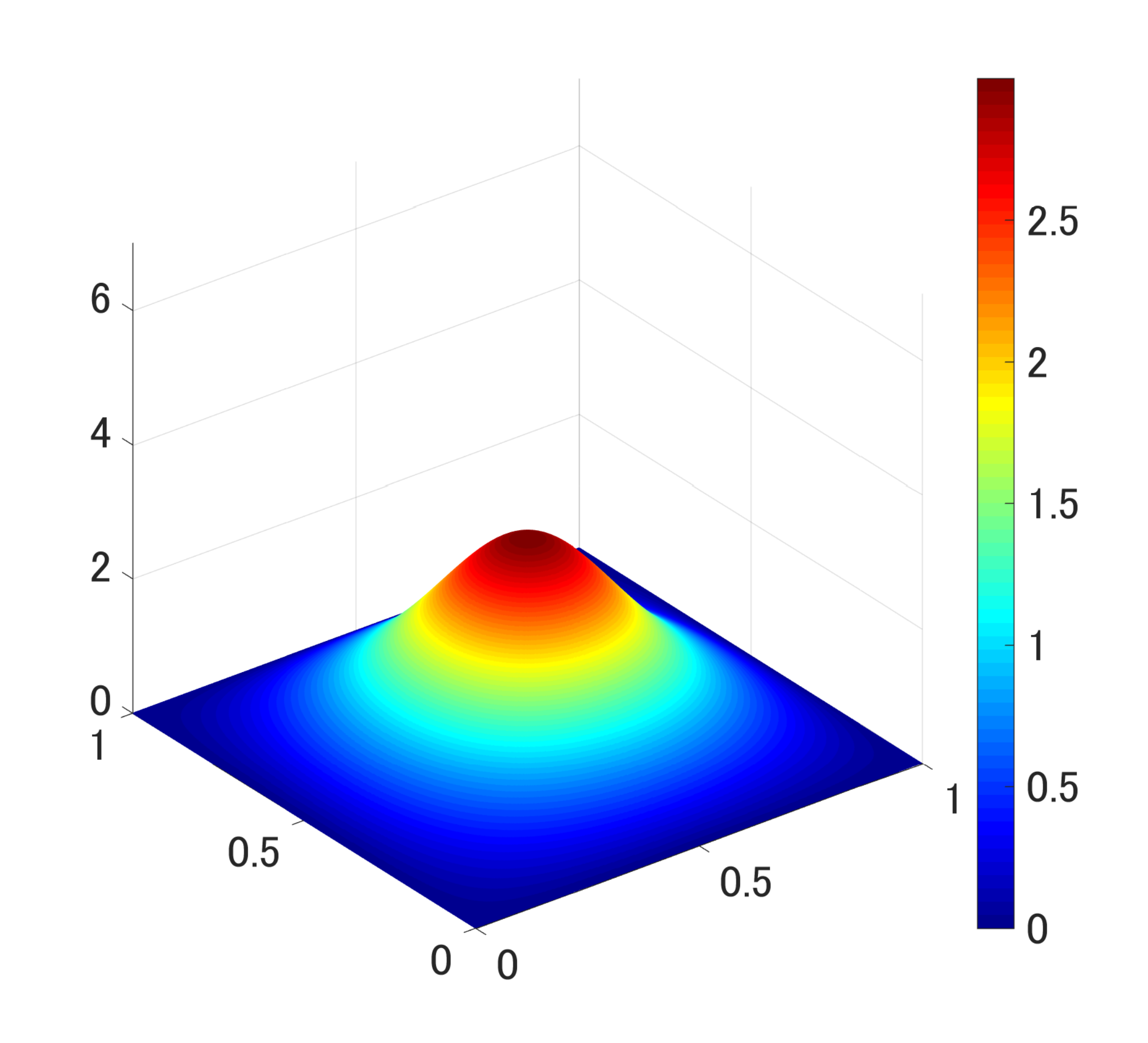}\\
				\footnotesize$p=5$,~~$\displaystyle \max_{x\in\Omega}\hat{u}(x)\approx 3.1721$
			\end{center}
			~
		\end{minipage}
		\caption{Approximate solutions to \eqref{emden} on $\Omega=(0,1)^{2}$ for $p=3, 5$.}
		\label{fig1}
	\end{figure}

	\begin{table}[H]
	\caption{Verification results for \eqref{emden} on $\Omega=(0,1)^{2}$ for $p=3, 5$. The values (except for those in row $ N $) represent strict upper bounds in decimal form.}
	\label{tab1}
	\begin{center}
		\renewcommand\arraystretch{1.3}
		\footnotesize
		\begin{tabular}{lll}
			\hline
			$p$& 3& 5\\
			\hline
			\hline
			$N$& 40 & 40 \\
			$ \|\mathcal{F}_{\hat{u}}^{\prime-1}\|_{\mathcal{L}\left(V^{*}, V\right)} $&1.70325176&2.36317681\\
			$ \|\mathcal{F}(\hat{u})\|_{V^{*}} $&$ 2.64173615\ten{-8} $&$ 1.92671579 \ten{-3} $\\
			$ L $&$ 0.67839778 $&$ 6.47198581 $\\
			$\alpha$& $ 4.49954173\ten{-8} $& $ 4.55317005\ten{-3} $\\
			$\beta$& $ 1.15548221 $& $ 15.2944468 $\\
			$\rho$& $ 4.63295216\ten{-8} $&	$ 5.47604979\ten{-3} $\\
			$C_{p+1}$&	$ 0.31830989 $& $ 0.39585400 $\\
			$ \left\|\hat{u}_{-}\right\|_{L^{p+1}} $&$ 4.19109326\ten{-2} $&$ 4.81952900\ten{-2} $\\
			$ C_{p+1}^{2}\left(\left\|\hat{u}_{-}\right\|_{L^{p+1}}+C_{p+1}\rho\right)^{p-1} $& $ 1.77973446\ten{-4} $&$ 1.00813027\ten{-6} $\\
			\hline
		\end{tabular}
	\end{center}
\end{table}
	
	In our next example, we consider the stationary problem of the Allen-Cahn equation
	\begin{align}
	\left\{\begin{array}{l l}
	-\Delta u=\varepsilon^{-2}(u-u^3) &\mathrm{in~} \Omega,\\
	u=0 &\mathrm{on~} \partial\Omega,
	\end{array}\right.\label{pro:allen}
	\end{align}
	where $\varepsilon>0$.
	We constructed approximate solutions $ \hat{u} $ of this problem using a Legendre polynomial basis in the same way, obtaining the figures displayed in Fig.~\ref{fig3}.
	The Lipschitz constant $ L $ was estimated as
	\begin{align*}
	L \leq 6 \varepsilon^{-2} C_{4}^{3}\left(\|\hat{u}\|_{L^{4}}+C_{4} r\right),~~r=2\alpha+\delta~~{\rm for~small}~~\delta>0
	\end{align*}
	in the same manner as \eqref{lip:emden} with $ p=3 $.
	Using Theorem \ref{theo:nk}, we again obtained $ H^1_0 $-error estimations for solutions of \eqref{pro:allen} centered around these approximations.
	Table \ref{tab3} shows the verification results for $\varepsilon=0.1$, $0.05$, and $0.025$.
	The positivity of the verified solutions was confirmed on the basis of Corollary \ref{cor:allen} (or Theorem \ref{theo:keeppositive}),
	where the required condition $\mu_1(\hat{u}) > 0$ was ensured in all cases.
	The lower bounds on $\mu_1(\hat{u})$ were computed numerically by estimating all rounding errors using the method in \cite{liu2015framework} with the interpolation-error constant $ C_N $ satisfying \eqref{eq:interpolation}.
	Note that proving the positivity of $ \hat{u} $ was also ignored in this example.

	\renewcommand{\sizee}{0.325\hsize}
	\begin{figure}[H]
		\begin{minipage}{\sizee}
			\begin{center}
				\vminus\includegraphics[height=36 mm]{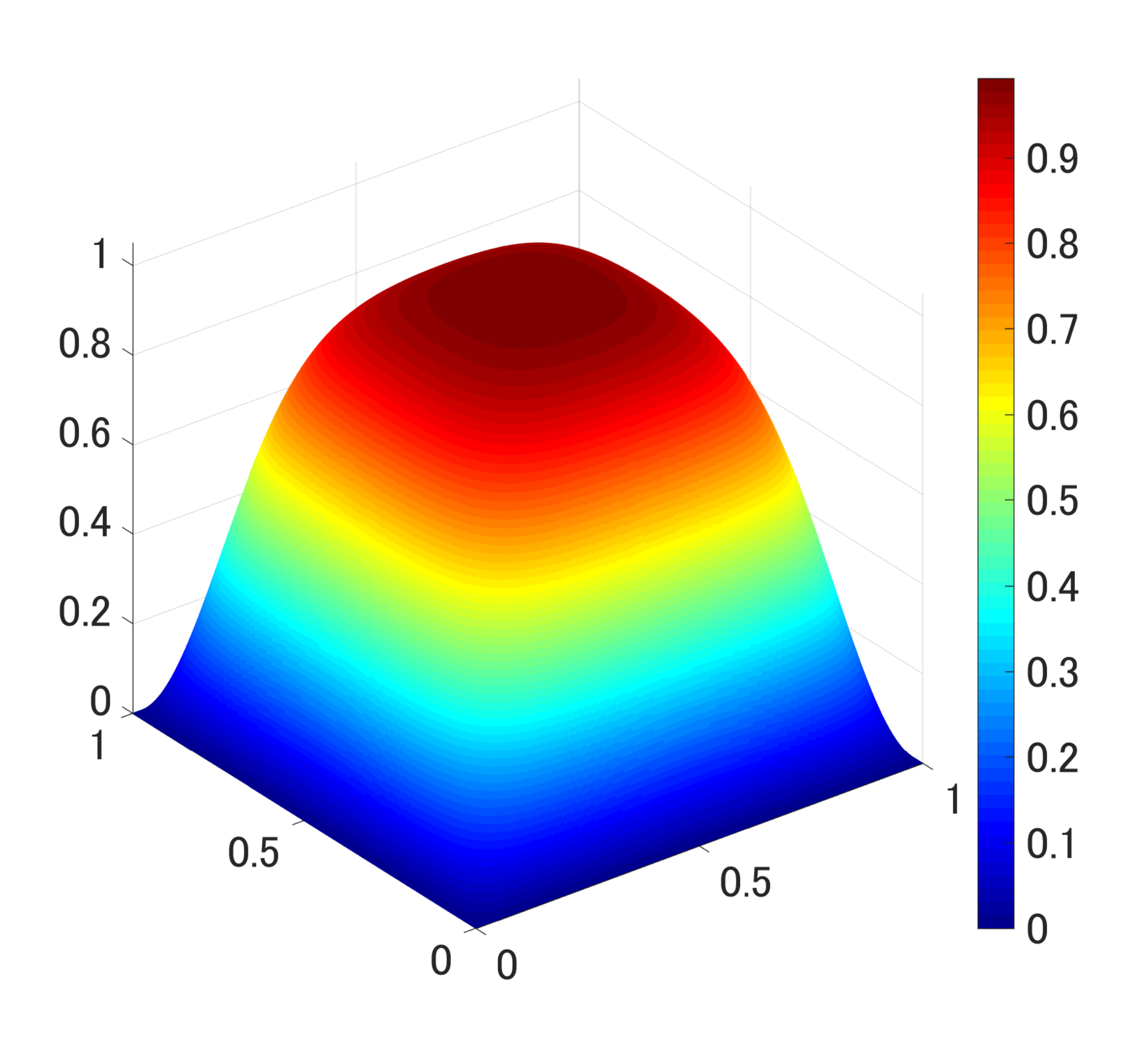}\\
				\footnotesize{$\varepsilon=0.1$}
			\end{center}
			~
		\end{minipage}
		\begin{minipage}{\sizee}
			\begin{center}
				\vminus\includegraphics[height=36 mm]{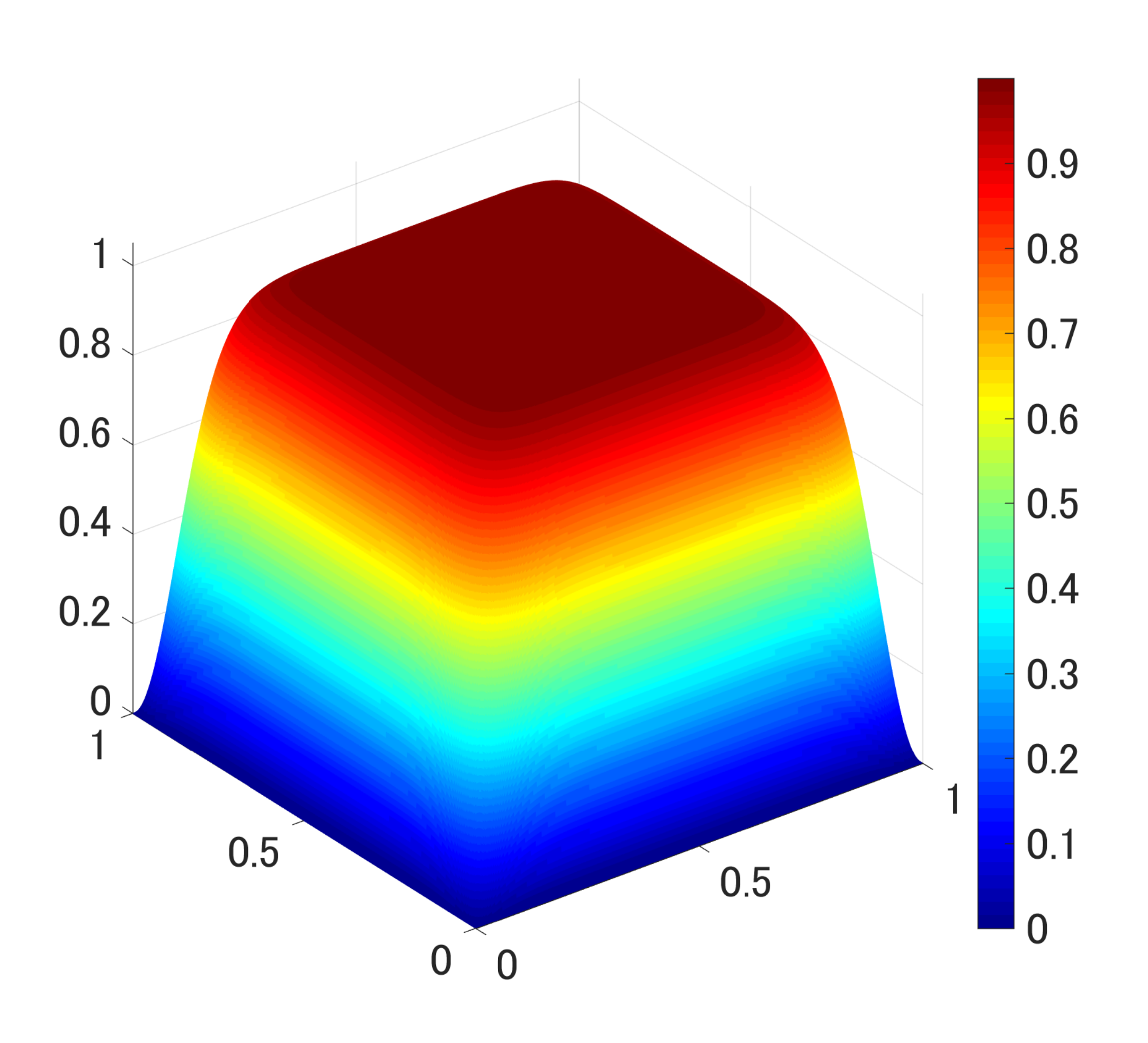}\\
				\footnotesize{$\varepsilon=0.05$}
			\end{center}
			~
		\end{minipage}
		\begin{minipage}{\sizee}
			\begin{center}
				\vminus\includegraphics[height=36 mm]{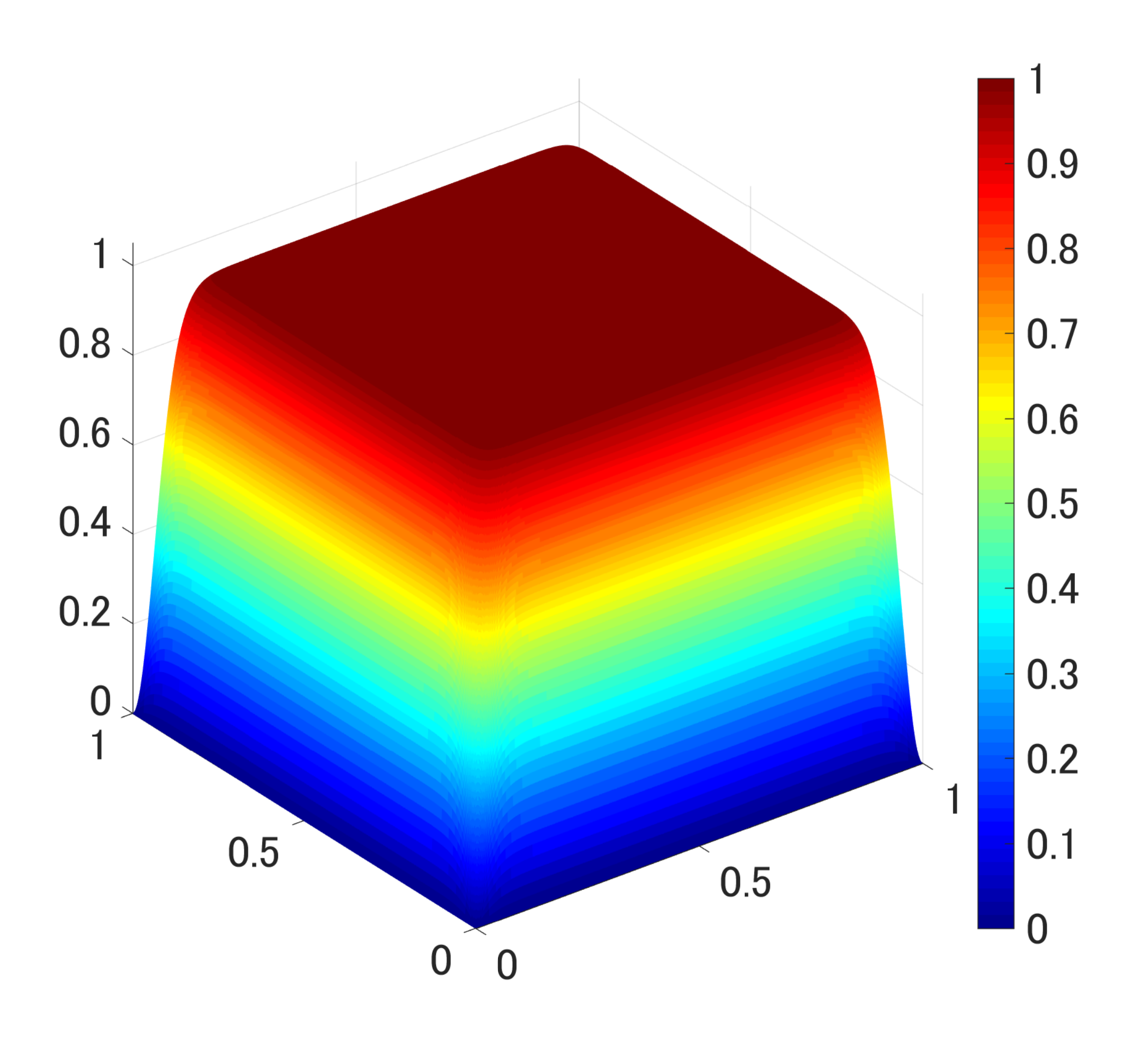}\\
				\footnotesize{$\varepsilon=0.025$}
			\end{center}
			~
		\end{minipage} 
		\caption{Approximate solutions to \eqref{pro:allen} on $\Omega=(0,1)^{2}$ for $\varepsilon=0.1$, $0.05$, and $0.025$.}
		\label{fig3}
	\end{figure} 
	
	\begin{table}[H]
		\caption{Verification results for \eqref{pro:allen} on $\Omega=(0,1)^{2}$ for $\varepsilon=0.1$, $0.05$, and $0.025$. The values (except for those in rows $ N $ and $\mu_1(\hat{u})$) represent strict upper bounds in decimal form. The values in row $\mu_1(\hat{u})$ represent strict lower bounds in decimal form.}
		\label{tab3}
		\begin{center}
		\renewcommand\arraystretch{1.3}
		\footnotesize
		\begin{tabular}{llll}
			\hline
			$\varepsilon$& 0.1& 0.05&0.025\\
			\hline
			\hline
			$N$& 40 & 40 & 60\\
			$ \|\mathcal{F}_{\hat{u}}^{\prime-1}\|_{\mathcal{L}\left(V^{*}, V\right)} $&$2.85871420$&$4.57367687$&$26.8239159$\\
			$ \|\mathcal{F}(\hat{u})\|_{V^{*}} $&$5.57390453\ten{-10}$&$2.15869521\ten{-6}$&$1.99428443\ten{-6}$\\
			$ L $&$3.00408573$&$5.02704780$&$7.57229904$\\
			$\alpha$&$ 1.59342000\ten{-9}$&$ 9.87317430\ten{-6}$&$5.34945174\ten{-5}$\\
			$\beta$&$ 8.58782250$&$ 22.9920923$&$2.03118712\ten{+2}$\\
			$\rho$&$ 1.59342002\ten{-9}$&$	9.87429519\ten{-6}$&$5.37883476\ten{-5}$\\
			$\mu_1(\hat{u}) \geq$&$ 1.13045870\ten{+2}$&$3.89094819\ten{+2}$&$1.18578735\ten{+3}$\\
			\hline
		\end{tabular}
		\end{center}
	\end{table}

	\section{Conclusion}
	In this paper, we have proposed methods for verifying the positivity of weak solutions $ u $ of the elliptic problem \eqref{mainpro:d} (namely, solutions of \eqref{main:fpro}) using the $ H^1_0 $-error estimation $ \left\|u-\hat{u}\right\|_{H_{0}^{1}} \leq \rho $ for some numerical approximation $ \hat{u} \in V $ and an explicit error bound $ \rho $.
	Theorem \ref{theo1} and \ref{theo:keeppositive} provide sufficient conditions for the solution $ u $ to be nonnegative.
	Using the maximum principle \cite{drabek2009maximum}, the positivity of $ u $ follows from the nonnegativity.
	Our theorems have a wide range of applications, including several important problems such as the elliptic problem \eqref{mainpro:d} with polynomial nonlinearities.
	Numerical examples confirmed that our approach works effectively for several important problems.
	\appendix
	\renewcommand{\thetheo}{\Alph{section}.\arabic{theo}}
	\section{The case $\lambda \geq \lambda_1(\Omega)$ and coefficients with different signs}
	In this section, we discuss the method for verifying positivity in the case where the approaches proposed in Sections \ref{sec:1} and \ref{sec:nk} are not applicable.
	This admits the nonlinearity of the form
	\begin{align}
		\label{ft:appendix}
		f(t)=\lambda t+\displaystyle \sum_{i=2}^{n(<p^*)}a_{i}t|t|^{i-1}~~\mbox{with}~~\lambda \geq \lambda_1(\Omega)~~\mbox{and~coefficients~satisfying}~~a_i a_j < 0~~\mbox{for~some~pair}~\{i,j\}.
	\end{align}
	However, this requires $ L^{\infty} $-norm estimation of the desired solution $ u \in V \cap L^{\infty}(\Omega) $, namely information about
	\begin{align}
	\label{linferror}
	\left\|u-\hat{u}\right\|_{L^{\infty}} \leq r
	\end{align}
	with $ \hat{u} \in V \cap L^{\infty}(\Omega) $ and $ r>0 $,
	which can be derived for highly-regular domains $ \Omega $ such as bounded convex polygonal domains using, for example, the method described in \cite{plum2001computer}.
	We define a subset $ \Omega_0 $ of $ \Omega $ where $ u $ may be negative by
	\begin{align*}
	\Omega_{0}=\text{Interior~of~}\left[\Omega\backslash\left\{x\in\Omega : \hat{u}-r>0\right\}\right].
	\end{align*}
	The following corollary is quite similar to Theorem \ref{theo1}.
	However, the assumption on $ \lambda $ is weakened while requiring evaluation of a lower bound for $ \lambda_{1}(\Omega_{0}) $.
	
	\begin{cor}\label{theo:linf}
		The same argument as used in Theorem $ \ref{theo1} $ follows by replacing $ \lambda_1(\Omega) $ with $ \lambda_1(\Omega_0) $.
	\end{cor}
	
	\begin{proof}
		The definition of $ \Omega_{0} $ ensures that the negative part $ u_{-} $ of the verified solution $ u $ belongs to $ H^1_0(\Omega_{0}) $.
		Therefore, replacing $ \lambda_1(\Omega) $ with $ \lambda_1(\Omega_0) $ in the proof of Theorem \ref{theo1} maintains the correctness of the proof.
	\end{proof}
	\begin{rem}
		In the same manner as mentioned in Remark $\ref{rem1:theo1}$,
		by extracting nonnegative coefficients $ a_i $, 
		the polynomial \eqref{ft:appendix} is confirmed to satisfy
		the required inequality \eqref{f:theo1}.
	\end{rem}
	\begin{rem}
		It is worth noting that $\lambda_1(\Omega'_0) \leq \lambda_1(\Omega_0) $ holds for a superset $ \Omega_0' \supset \Omega_{0} $.
		Therefore, even when $ \Omega_0$ has a complicated shape, one only has to estimate the lower bound for $\lambda_1(\Omega'_0)$ on such a superset $ \Omega_0'$ with a simple shape as long as $\lambda < \lambda_1(\Omega'_0)$.
		A lower bound for the eigenvalue $ \lambda_1(\Omega_0') $ can be numerically evaluated using the method, for example, in {\rm \cite{liu2013verified}} with a suitable basis that spans the functions over $ \Omega_{0} $, such as the finite element basis.
	\end{rem}
	
	\begin{rem}
		The range of application of Corollary {\rm \ref{theo:linf}} covers all the cases listed in Table $ \ref{table:applicable} $, as long as we have accurate $ L^{\infty} $-norm estimation as in \eqref{linferror} and can evaluate a lower bound for $ \lambda_1(\Omega_0) $ satisfying the required inequality. 
	\end{rem}

	\section*{Acknowledgments}
	We express our sincere thanks to Professor Kazunaga Tanaka (Waseda University, Japan) for helpful advice and comments about this study,
	Professor Michael Plum (Karlsruhe Institut f$\ddot{\rm u}$r Technologie, Germany) for helping us to correct a mistake in Theorem \ref{theo1},
	and Kohei Yatabe (Waseda University, Japan) for his contribution to improving English expressions of this paper.
	We also express our profound gratitude to two anonymous referees for their highly insightful comments and suggestions.
	This work is supported by JST CREST Grant Numbers JPMJCR14D4, and JSPS KAKENHI Grant Number JP17H07188 and JP19K14601, and Mizuho Foundation for the Promotion of Sciences.
	\bibliography{ref}

\begin{thebibliography}{10}
\expandafter\ifx\csname url\endcsname\relax
  \def\url#1{\texttt{#1}}\fi
\expandafter\ifx\csname urlprefix\endcsname\relax\def\urlprefix{URL }\fi
\expandafter\ifx\csname href\endcsname\relax
  \def\href#1#2{#2} \def\path#1{#1}\fi

\bibitem{lions1982existence}
P.-L. Lions, On the existence of positive solutions of semilinear elliptic
  equations, SIAM review 24~(4) (1982) 441--467.

\bibitem{gidas1979symmetry}
B.~Gidas, W.-M. Ni, L.~Nirenberg, Symmetry and related properties via the
  maximum principle, Communications in Mathematical Physics 68~(3) (1979)
  209--243.

\bibitem{lin1994uniqueness}
C.-S. Lin, Uniqueness of least energy solutions to a semilinear elliptic
  equation in $\mathbb{R}^{2}$, manuscripta mathematica 84~(1) (1994) 13--19.

\bibitem{damascelli1999qualitative}
L.~Damascelli, M.~Grossi, F.~Pacella, Qualitative properties of positive
  solutions of semilinear elliptic equations in symmetric domains via the
  maximum principle, Annales de l'Institut Henri Poincare-Nonlinear Analysis
  16~(5) (1999) 631--652.

\bibitem{gladiali2011bifurcation}
F.~Gladiali, M.~Grossi, F.~Pacella, P.~Srikanth, Bifurcation and symmetry
  breaking for a class of semilinear elliptic equations in an annulus, Calculus
  of Variations and Partial Differential Equations 40~(3) (2011) 295--317.

\bibitem{de2019morse}
F.~De~Marchis, M.~Grossi, I.~Ianni, F.~Pacella, Morse index and uniqueness of
  positive solutions of the lane-emden problem in planar domains, Journal de
  Math{\'e}matiques Pures et Appliqu{\'e}es (2019) in press.

\bibitem{allen1979microscopic}
S.~M. Allen, J.~W. Cahn, A microscopic theory for antiphase boundary motion and
  its application to antiphase domain coarsening, Acta Metallurgica 27~(6)
  (1979) 1085--1095.

\bibitem{nakao1988numerical}
M.~T. Nakao, A numerical approach to the proof of existence of solutions for
  elliptic problems, Japan Journal of Applied Mathematics 5~(2) (1988)
  313--332.

\bibitem{plum1991computer}
M.~Plum, Computer-assisted existence proofs for two-point boundary value
  problems, Computing 46~(1) (1991) 19--34.

\bibitem{plum2008}
M.~Plum, Existence and multiplicity proofs for semilinear elliptic boundary
  value problems by computer assistance, Jahresbericht der Deutschen
  Mathematiker Vereinigung 110~(1) (2008) 19--54.

\bibitem{nakao2011numerical}
M.~T. Nakao, Y.~Watanabe, Numerical verification methods for solutions of
  semilinear elliptic boundary value problems, Nonlinear Theory and Its
  Applications, IEICE 2~(1) (2011) 2--31.

\bibitem{tanaka2017sharp}
K.~Tanaka, K.~Sekine, M.~Mizuguchi, S.~Oishi, {Sharp numerical inclusion of the
  best constant for embedding $H_0^1(\Omega) \hookrightarrow L^p(\Omega)$ on
  bounded convex domain}, Journal of Computational and Applied Mathematics 311
  (2017) 306--313.

\bibitem{tanaka2015numerical}
K.~Tanaka, K.~Sekine, M.~Mizuguchi, S.~Oishi, Numerical verification of
  positiveness for solutions to semilinear elliptic problems, JSIAM Letters 7
  (2015) 73--76.

\bibitem{tanaka2017numerical}
K.~Tanaka, K.~Sekine, S.~Oishi, Numerical verification method for positivity of
  solutions to elliptic equations, RIMS K{\^o}ky{\^u}roku 2037 (2017) 117--125.

\bibitem{drabek2009maximum}
P.~Dr{\'a}bek, On a maximum principle for weak solutions of some quasi-linear
  elliptic equations, Applied Mathematics Letters 22~(10) (2009) 1567--1570.

\bibitem{deuflhard1979affine}
P.~Deuflhard, G.~Heindl, Affine invariant convergence theorems for newton’s
  method and extensions to related methods, SIAM Journal on Numerical Analysis
  16~(1) (1979) 1--10.

\bibitem{kantorovich1982functional}
L.~V. Kantorovich, G.~P. Akilov, Functional Analysis, Pergamon Press, Oxford,
  1982.

\bibitem{argyros2008convergence}
I.~K. Argyros, Convergence and applications of Newton-type iterations, Springer
  Science \& Business Media, 2008.

\bibitem{proinov2010new}
P.~D. Proinov, New general convergence theory for iterative processes and its
  applications to newton--kantorovich type theorems, Journal of Complexity
  26~(1) (2010) 3--42.

\bibitem{argyros2012weaker}
I.~K. Argyros, S.~Hilout, Weaker conditions for the convergence of newton’s
  method, Journal of Complexity 28~(3) (2012) 364--387.

\bibitem{rump1999book}
S.~Rump, {INTLAB - INTerval LABoratory}, in: T.~Csendes (Ed.),
  {Developments~in~Reliable Computing}, Kluwer Academic Publishers, Dordrecht,
  1999, pp. 77--104, \url{http://www.ti3.tuhh.de/rump/}.

\bibitem{kashiwagikv}
M.~Kashiwagi, {kv library}, \url{http://verifiedby.me/kv/} (2019).

\bibitem{kimura1999on}
S.~Kimura, N.~Yamamoto, On explicit bounds in the error for the $ {H_0^1}
  $-projection into piecewise polynomial spaces, Bulletin of informatics and
  cybernetics 31~(2) (1999) 109--115.

\bibitem{tanaka2014verified}
K.~Tanaka, A.~Takayasu, X.~Liu, S.~Oishi, Verified norm estimation for the
  inverse of linear elliptic operators using eigenvalue evaluation, Japan
  Journal of Industrial and Applied Mathematics 31~(3) (2014) 665--679.

\bibitem{liu2015framework}
X.~Liu, A framework of verified eigenvalue bounds for self-adjoint differential
  operators, Applied Mathematics and Computation 267 (2015) 341--355.

\bibitem{plum2001computer}
M.~Plum, Computer-assisted enclosure methods for elliptic differential
  equations, Linear Algebra and its Applications 324~(1) (2001) 147--187.

\bibitem{liu2013verified}
X.~Liu, S.~Oishi, Verified eigenvalue evaluation for the laplacian over
  polygonal domains of arbitrary shape, SIAM Journal on Numerical Analysis
  51~(3) (2013) 1634--1654.

\end{thebibliography}
	
\end{document}